\documentclass[11pt]{article}
\usepackage{amssymb, amsmath, fullpage, amsthm}
\usepackage{mathrsfs}
\usepackage{tikz}
\usepackage[title]{appendix}

\newtheorem{theorem}{Theorem}[section]
\newtheorem{heuristic}[theorem]{Heuristic Result}
\newtheorem{question}[theorem]{Question}

\DeclareMathOperator{\li}{li}

\numberwithin{equation}{section}

\newcommand{\vanish}[1]{}

\begin{document}

\title{Gauss Circle Primes}

\author{\sc Thomas  EHRENBORG
}

\date{}

\maketitle

\begin{abstract}
Given a circle of radius $r$ centered at the origin,
the Gauss Circle Problem concerns
counting the number of lattice points $C(r)$ within
this  circle.
It is known
that as $r$ grows large, the number of lattice points approaches
$\pi r^2$, that is,
the area of the circle.
The present research is to
study how often $C(r)$ will return a prime
number of lattice points for $r \leq n$.
The Prime Number Theorem
predicts that the number of primes
less than 
or equal to $n$ 
is asymptotic to 
$\frac{n}{\log n}$.
We find that the number  of Gauss Circle Primes
for $r \leq n$ is also of order 
$\frac{n}{\log n}$
for $n \leq 2 \times 10^6$.
We include a heuristic argument that the
Gauss Circle Primes can be approximated
by $\frac{n}{\log n}$.
\end{abstract}

\section{Gauss Circle Problem}

The Gauss Circle Problem
concerns finding the number of lattice points
$C(r)$ within a circle centered at the origin
of radius $r$.
As $r$ increases, the number of lattice
points $C(r)$ approaches $\pi r^2$, the
area of the circle.
See Figure~\ref{figure_example_radius_5} for the example
when $r = 5$.
Beginning with Gauss in 1834~\cite[pages 271 and 277]{Gauss},
researchers have focused on
improving the
error bounds between $C(r)$ and the area
$\pi r^2$.

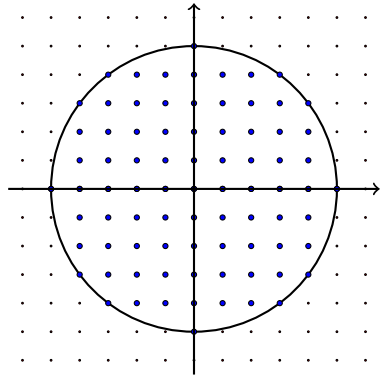
\begin{figure}
\begin{center}

\begin{tikzpicture}[scale = 0.38]
\foreach \x in {-6,-5, ..., 6}
{\foreach \y in {-6,-5, ..., 6}
{\draw[fill=red] (\x,\y) circle (0.3mm);
};};

\draw[thick] (0,0) circle (5);

{\foreach \y in {0,1, ..., 5}
{
\draw[fill=blue] (0,\y) circle (0.9mm);
\draw[fill=blue] (0,-\y) circle (0.9mm);
};};

{\foreach \y in {0,1, ..., 4}
{
\draw[fill=blue] (1,\y) circle (0.9mm);
\draw[fill=blue] (1,-\y) circle (0.9mm);
\draw[fill=blue] (-1,\y) circle (0.9mm);
\draw[fill=blue] (-1,-\y) circle (0.9mm);
};};

{\foreach \y in {0,1, ..., 4}
{
\draw[fill=blue] (2,\y) circle (0.9mm);
\draw[fill=blue] (2,-\y) circle (0.9mm);
\draw[fill=blue] (-2,\y) circle (0.9mm);
\draw[fill=blue] (-2,-\y) circle (0.9mm);
};};

{\foreach \y in {0,1, ..., 4}
{
\draw[fill=blue] (3,\y) circle (0.9mm);
\draw[fill=blue] (3,-\y) circle (0.9mm);
\draw[fill=blue] (-3,\y) circle (0.9mm);
\draw[fill=blue] (-3,-\y) circle (0.9mm);
};};

{\foreach \y in {0,1, ..., 3}
{
\draw[fill=blue] (4,\y) circle (0.9mm);
\draw[fill=blue] (4,-\y) circle (0.9mm);
\draw[fill=blue] (-4,\y) circle (0.9mm);
\draw[fill=blue] (-4,-\y) circle (0.9mm);
};};

{\foreach \y in {0}
{
\draw[fill=blue] (5,\y) circle (0.9mm);
\draw[fill=blue] (5,-\y) circle (0.9mm);
\draw[fill=blue] (-5,\y) circle (0.9mm);
\draw[fill=blue] (-5,-\y) circle (0.9mm);
};};

\draw[->,thick] (0,-6.5) -- (0,6.5);
\draw[->,thick] (-6.5,0) -- (6.5,0);

\end{tikzpicture}
\end{center}
\caption{Lattice points on a circle of radius $5$ showing $C(5) = 81$.}
\label{figure_example_radius_5} 
\end{figure}

\begin{theorem}
[Gauss 1834]
The error bound between the number of lattice points
$C(r)$ within a circle of radius $r$ centered at the origin
and the area of the circle is given by
\begin{equation}
     |C(r) - \pi r^2 | < 2 \sqrt{2} \pi r + 2 \pi.
\label{equation_Gauss}
\end{equation}
\end{theorem}
\begin{proof}
For each lattice point within a circle of radius $r$,
we put a unit square to the upper right of each lattice point.
See Figure~\ref{figure_geometric_argument} for the region
in the case 
when $r=5$.
Since each unit square has area $1$, 
by bounding the area of the shape created
by the squares, we simultaneously
bound the number of points within the circle.

We know that the maximum distance that adding 
the unit squares can excede 
the boundary
of the circle is $\sqrt{2}$ because 
the diagonal length of a single unit square is
$\sqrt{2}$.
A similar argument can be used 
to find a lower bound for the area.
Thus we can bound the area of the shape with 
two circles of radii
$r + \sqrt{2}$ and $r - \sqrt{2}$, that is,
the 
area of the inner circle is less than the area of the 
squares  
is less than the  area of the outer circle:
$$
   \pi (r - \sqrt{2})^2 < C(r)  <    \pi (r + \sqrt{2})^2 .
$$
Expanding and subtracting $\pi r^2$ gives
\begin{equation}
  - 2 \sqrt{2} \pi r + 2 \pi < C(r) - \pi r^2   
           <    2 \sqrt{2} \pi r + 2 \pi.
\label{equation_koala}
\end{equation}
Since $- (2 \sqrt{2} \pi r + 2 \pi) < -2 \sqrt{2} \pi r + 2 \pi$,
the inequality in~(\ref{equation_koala}) 
implies (\ref{equation_Gauss}), as claimed.\end{proof}

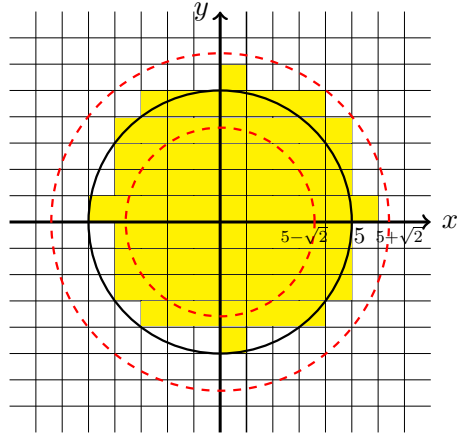
\begin{figure}
\begin{center}
\begin{tikzpicture}[x=3.5mm, y=3.5mm]

\draw[line width = 0.2mm] (1,8) -- (1,-8);
\draw[line width = 0.1mm] (2,8) -- (2,-8);
\draw[line width = 0.1mm] (3,8) -- (3,-8);
\draw[line width = 0.1mm] (4,8) -- (4,-8);
\draw[line width = 0.1mm] (5,8) -- (5,-8);
\draw[line width = 0.1mm] (6,8) -- (6,-8);
\draw[line width = 0.1mm] (7,8) -- (7,-8);
\draw[line width = 0.1mm] (-1,8) -- (-1,-8);
\draw[line width = 0.1mm] (-2,8) -- (-2,-8);
\draw[line width = 0.1mm] (-3,8) -- (-3,-8);
\draw[line width = 0.1mm] (-4,8) -- (-4,-8);
\draw[line width = 0.1mm] (-5,8) -- (-5,-8);
\draw[line width = 0.1mm] (-6,8) -- (-6,-8);
\draw[line width = 0.1mm] (-7,8) -- (-7,-8);

\draw[line width = 0.1mm] (-8,1) -- (8,1);
\draw[line width = 0.1mm] (-8,2) -- (8,2);
\draw[line width = 0.1mm] (-8,3) -- (8,3);
\draw[line width = 0.1mm] (-8,4) -- (8,4);
\draw[line width = 0.1mm] (-8,5) -- (8,5);
\draw[line width = 0.1mm] (-8,6) -- (8,6);
\draw[line width = 0.1mm] (-8,7) -- (8,7);

\draw[line width = 0.1mm] (-8,-1) -- (8,-1);
\draw[line width = 0.1mm] (-8,-2) -- (8,-2);
\draw[line width = 0.1mm] (-8,-3) -- (8,-3);
\draw[line width = 0.1mm] (-8,-4) -- (8,-4);
\draw[line width = 0.1mm] (-8,-5) -- (8,-5);
\draw[line width = 0.1mm] (-8,-6) -- (8,-6);
\draw[line width = 0.1mm] (-8,-7) -- (8,-7);

\draw [draw=none, fill=yellow] (0.04,0.04) rectangle(.96,.96);
\draw [draw=none, fill=yellow] (1.02,0.02) rectangle(1.97,.97);
\draw [draw=none, fill=yellow] (2.02,0.02) rectangle(2.97,.97);
\draw [draw=none, fill=yellow] (3.02,0.02) rectangle(3.97,.97);
\draw [draw=none, fill=yellow] (4.02,0.02) rectangle(4.97,.97);
\draw [draw=none, fill=yellow] (5.02,0.02) rectangle(5.97,.97);
\draw [draw=none, fill=yellow] (-0.02,0.02) rectangle(-.97,.97);
\draw [draw=none, fill=yellow] (-1.02,0.02) rectangle(-1.97,.97);
\draw [draw=none, fill=yellow] (-2.02,0.02) rectangle(-2.97,.97);
\draw [draw=none, fill=yellow] (-3.02,0.02) rectangle(-3.97,.97);
\draw [draw=none, fill=yellow] (-4.02,0.02) rectangle(-4.97,.97);

\draw [draw=none, fill=yellow] (0.04,-0.02) rectangle(.97,-.97);
\draw [draw=none, fill=yellow] (1.02,-0.02) rectangle(1.97,-.97);
\draw [draw=none, fill=yellow] (2.02,-0.02) rectangle(2.97,-.97);
\draw [draw=none, fill=yellow] (3.02,-0.02) rectangle(3.97,-.97);
\draw [draw=none, fill=yellow] (4.02,-0.02) rectangle(4.97,-.97);
\draw [draw=none, fill=yellow] (-0.02,-0.02) rectangle(-.97,-.97);
\draw [draw=none, fill=yellow] (-1.02,-0.02) rectangle(-1.97,-.97);
\draw [draw=none, fill=yellow] (-2.02,-0.02) rectangle(-2.97,-.97);
\draw [draw=none, fill=yellow] (-3.02,-0.02) rectangle(-3.97,-.97);

\draw [draw=none, fill=yellow] (0.02,1.02) rectangle(.96,1.97);
\draw [draw=none, fill=yellow] (1.02,1.02) rectangle(1.97,1.97);
\draw [draw=none, fill=yellow] (2.02,1.02) rectangle(2.97,1.97);
\draw [draw=none, fill=yellow] (3.02,1.02) rectangle(3.97,1.97);
\draw [draw=none, fill=yellow] (4.02,1.02) rectangle(4.97,1.97);
\draw [draw=none, fill=yellow] (-0.02,1.02) rectangle(-.97,1.97);
\draw [draw=none, fill=yellow] (-1.02,1.02) rectangle(-1.97,1.97);
\draw [draw=none, fill=yellow] (-2.02,1.02) rectangle(-2.97,1.97);
\draw [draw=none, fill=yellow] (-3.02,1.02) rectangle(-3.97,1.97);

\draw [draw=none, fill=yellow] (0.02,2.02) rectangle(.97,2.97);
\draw [draw=none, fill=yellow] (1.02,2.02) rectangle(1.97,2.97);
\draw [draw=none, fill=yellow] (2.02,2.02) rectangle(2.97,2.97);
\draw [draw=none, fill=yellow] (3.02,2.02) rectangle(3.97,2.97);
\draw [draw=none, fill=yellow] (4.02,2.02) rectangle(4.97,2.97);
\draw [draw=none, fill=yellow] (-0.02,2.02) rectangle(-.97,2.97);
\draw [draw=none, fill=yellow] (-1.02,2.02) rectangle(-1.97,2.97);
\draw [draw=none, fill=yellow] (-2.02,2.02) rectangle(-2.97,2.97);
\draw [draw=none, fill=yellow] (-3.02,2.02) rectangle(-3.97,2.97);

\draw [draw=none, fill=yellow] (0.02,3.02) rectangle(.97,3.97);
\draw [draw=none, fill=yellow] (1.02,3.02) rectangle(1.97,3.97);
\draw [draw=none, fill=yellow] (2.02,3.02) rectangle(2.97,3.97);
\draw [draw=none, fill=yellow] (3.02,3.02) rectangle(3.97,3.97);
\draw [draw=none, fill=yellow] (4.02,3.02) rectangle(4.97,3.97);
\draw [draw=none, fill=yellow] (-0.02,3.02) rectangle(-.97,3.97);
\draw [draw=none, fill=yellow] (-1.02,3.02) rectangle(-1.97,3.97);
\draw [draw=none, fill=yellow] (-2.02,3.02) rectangle(-2.97,3.97);
\draw [draw=none, fill=yellow] (-3.02,3.02) rectangle(-3.97,3.97);

\draw [draw=none, fill=yellow] (0.02,4.02) rectangle(.97,4.97);
\draw [draw=none, fill=yellow] (1.02,4.02) rectangle(1.97,4.97);
\draw [draw=none, fill=yellow] (2.02,4.02) rectangle(2.97,4.97);
\draw [draw=none, fill=yellow] (3.02,4.02) rectangle(3.97,4.97);
\draw [draw=none, fill=yellow] (-0.02,4.02) rectangle(-.97,4.97);
\draw [draw=none, fill=yellow] (-1.02,4.02) rectangle(-1.97,4.97);
\draw [draw=none, fill=yellow] (-2.02,4.02) rectangle(-2.97,4.97);


\draw [draw=none, fill=yellow] (0.02,-1.02) rectangle(.96,-1.97);
\draw [draw=none, fill=yellow] (1.02,-1.02) rectangle(1.97,-1.97);
\draw [draw=none, fill=yellow] (2.02,-1.02) rectangle(2.97,-1.97);
\draw [draw=none, fill=yellow] (3.02,-1.02) rectangle(3.97,-1.97);
\draw [draw=none, fill=yellow] (4.02,-1.02) rectangle(4.97,-1.97);
\draw [draw=none, fill=yellow] (-0.02,-1.02) rectangle(-.97,-1.97);
\draw [draw=none, fill=yellow] (-1.02,-1.02) rectangle(-1.97,-1.97);
\draw [draw=none, fill=yellow] (-2.02,-1.02) rectangle(-2.97,-1.97);
\draw [draw=none, fill=yellow] (-3.02,-1.02) rectangle(-3.97,-1.97);

\draw [draw=none, fill=yellow] (0.02,-2.02) rectangle(.97,-2.97);
\draw [draw=none, fill=yellow] (1.02,-2.02) rectangle(1.97,-2.97);
\draw [draw=none, fill=yellow] (2.02,-2.02) rectangle(2.97,-2.97);
\draw [draw=none, fill=yellow] (3.02,-2.02) rectangle(3.97,-2.97);
\draw [draw=none, fill=yellow] (4.02,-2.02) rectangle(4.97,-2.97);
\draw [draw=none, fill=yellow] (-0.02,-2.02) rectangle(-.97,-2.97);
\draw [draw=none, fill=yellow] (-1.02,-2.02) rectangle(-1.97,-2.97);
\draw [draw=none, fill=yellow] (-2.02,-2.02) rectangle(-2.97,-2.97);
\draw [draw=none, fill=yellow] (-3.02,-2.02) rectangle(-3.97,-2.97);

\draw [draw=none, fill=yellow] (0.02,-3.02) rectangle(.97,-3.97);
\draw [draw=none, fill=yellow] (1.02,-3.02) rectangle(1.97,-3.97);
\draw [draw=none, fill=yellow] (2.02,-3.02) rectangle(2.97,-3.97);
\draw [draw=none, fill=yellow] (3.02,-3.02) rectangle(3.97,-3.97);
\draw [draw=none, fill=yellow] (-0.02,-3.02) rectangle(-.97,-3.97);
\draw [draw=none, fill=yellow] (-1.02,-3.02) rectangle(-1.97,-3.97);
\draw [draw=none, fill=yellow] (-2.02,-3.02) rectangle(-2.97,-3.97);

\draw [draw=none, fill=yellow] (0.02,4.02) rectangle(.97,4.97);
\draw [draw=none, fill=yellow] (1.02,4.02) rectangle(1.97,4.97);
\draw [draw=none, fill=yellow] (2.02,4.02) rectangle(2.97,4.97);
\draw [draw=none, fill=yellow] (3.02,4.02) rectangle(3.97,4.97);
\draw [draw=none, fill=yellow] (-0.02,4.02) rectangle(-.97,4.97);
\draw [draw=none, fill=yellow] (-1.02,4.02) rectangle(-1.97,4.97);
\draw [draw=none, fill=yellow] (-2.02,4.02) rectangle(-2.97,4.97);

\draw [draw=none, fill=yellow] (.02,5.02) rectangle(.97,5.97);

\draw [draw=none, fill=yellow] (0.02,-4.02) rectangle(0.97,-4.97);
\draw [draw=none, fill=yellow] (0.02,-4.02) rectangle(0.97,-4.97);

\draw[->,line width = 0.4mm] (-8,0) -- (8,0) node[right, scale=1.0]{$x$};
\draw[->,line width = 0.4mm] (0,-8) -- (0,8) node[left, scale=1.0]{$y$};


\draw[line width = 0.3mm] (0,0)  circle(5);
\draw[color=red, dashed, line width = 0.3mm] (0,0)  circle({5+sqrt(2)});
\draw[color=red, dashed, line width = 0.3mm] (0,0)  circle({5-sqrt(2)});


\draw ({5+0.3},{0-0.5}) node[scale=0.8]{$5$};
\draw ({5+sqrt(2)+0.39},{0-0.5}) node[scale=0.8]{$\scriptstyle 5+\sqrt{2}$};
\draw ({5-sqrt(2)-0.39},{0-0.5}) node[scale=0.8]{$\scriptstyle 5-\sqrt{2}$};


\end{tikzpicture}
\end{center}
\caption{Geometric argument for Gauss' error bound.}
\label{figure_geometric_argument}
\end{figure}

Sharper bounds for the error estimate 
are described using a circle of radius $\sqrt{x}$, rather
than an integer radius.
See~\cite{Iwaniec_Mozzochi} for further details.

\begin{theorem}

\vspace{.1in}
For a circle of radius $\sqrt{x}$
\begin{eqnarray*}
    |C(\sqrt{x}) - \pi x| = \left\{
     \begin{array}{ll}
       O(\sqrt{x})   &\mbox{[Gauss 1834]}\\
       O(x^{1/3})   &\mbox{[Sierpi\'nski 1906]}\\
       O(x^{37/112})
    &\mbox{[Littlewood--Walfisz 1927]}
\\
       O(x^{15/46})
    &\mbox{[Titchmarsh 1935]}
\\
       O(x^{13/40})
    &\mbox{[Hua 1942]}
\\
       O(x^{7/22})
    &\mbox{[Iwaniec--Mozzochi 1988]}
\end{array}
   \right.
\end{eqnarray*}
\end{theorem}

It is interesting to note that all of the bounds since
Littlewood-Walfisz and Landau are of the form
$O(x^{k/(3k+1)})$.

In 1988 Iwaniec and Mozzochi~\cite{Iwaniec_Mozzochi}
showed the conjectured bound $O(x^{1/4 + \epsilon})$
for $\epsilon > 0$ is related to the Riemann
hypothesis.

In this paper, we will instead focus on the arithmetic
properties of the values of $C(r)$. 
We call a value of $C(r)$ 
that is a prime number a {\em Gauss Circle Prime}.
We ask
how often will $C(r)$ return a prime number of
lattice points for $r \leq n$
where $r$ is an integer.
See Table~\ref{table_values}
for values of $C(r)$ for $r = 1, \ldots, 10$.

\begin{table}
\begin{center}
\begin{tabular}{r||r|r|r|r|r|r|r|r|r|r}
$r$   
&1
&2
&3
&4
&5
&6
&7
&8
&9
&10\\
\hline
$C(r)$
&{\color{blue}\bf 5}   
&{\color{blue}\bf 13}  
&{\color{blue}\bf 29}  
&49 
&81 
&{\color{blue}\bf 113}
&{\color{blue}\bf 149}  
&{\color{blue}\bf 197}
&253 
&{\color{blue}\bf 317} 
\end{tabular}
\end{center}
\label{table_values}
\caption{Values of $C(r)$ for $r = 1, \ldots, 10$.
Prime values are indicated in {\bf {\color{blue} blue}}.}
\end{table}

In order to state our results,
we recall the Prime Number Theorem;
see~\cite{Hadamard,delaValeePoussin}.
Here 
$\pi(n)$ is the number of primes less than or equal to $n$.

\begin{theorem}
[Hadamard and de la Vall\'ee Poussin, 1896]
The prime number function
$\pi(n)$ is asymptotic to $\frac{n}{\log n}$.
In other words, the following
limit holds:
$$
   \lim_{n \rightarrow \infty} \frac{\pi(n)}{\frac{n}{\log n}} = 1.
$$
\end{theorem}

Analogous to the prime number function,
let
$\kappa(n)$ be the number of Gauss Circle Primes for $r \leq n$.
We will focus on the following two questions.

\begin{question}
Do the Gauss Circle Primes, which have a geometric description, behave
like the Prime Number Theorem?
\end{question}

\begin{question}
More generally, are the Gauss Circle Primes and prime numbers
distributed similarly?
\end{question}

Understanding this new geometric model for prime numbers could give new
ideas in important open problems related to prime numbers
and applications
including:
(i) the Riemann Hypothesis,
which is related to the distribution of prime numbers:
All of the non-trivial zeroes of
the Riemann zeta function lie 
in the complex plane with real part equal to $1/2$,
(ii)
designing new encryption security that uses prime numbers, such as RSA,
and
(iii)
finding geometric methods to exploit current encryption methods.

\section{Preliminary Results and Data}

We begin with a result that follows from the symmetry of the lattice
points.  This is used to create a Java program that runs in linear
time in the variable $r$.

\begin{theorem}
  The number of lattice points in a circle of radius
  $r$ satisfies
  $$
     C(r) \equiv 1 \bmod 4.
  $$
In particular, if $C(r)$ is a Gauss Circle Prime
then $C(r) \equiv 1 \bmod 4$.
\label{theorem_number_lattice_points}
\end{theorem}
\begin{proof}
To count the number of lattice points
in a circle of radius $r$,
one must first count the number of lattice points
in the open first quadrant.
By symmetry, each of the four open quadrants contain
the same number of lattice points.
It remains to count the lattice points on the axes.
Ignoring the origin,
the positive $x$ axis has $r$ lattice points in the circle.
By symmetry, the negative $x$ axis, positive $y$ axis
and negative $y$ axis also each have $r$ lattice points
within the circle.
Adding the lattice point at the origin gives
$4r + 1$ lattice points on two axes.
Since the number of lattice
points in the four open quadrants is a multiple
of four,
and there are $4r+1$ lattice points on the axes,
the total number of lattice points in the circle
is of  the form $4k+1$ for some integer $k$.
\end{proof}

See Tables~\ref{table_one} and 3 for values
of $\pi(n)$,
$\kappa(n)$ and $\frac{n}{\log n}$.
Beginning at $n = 167$, the values satisfy the inequality
$\pi(n) > \kappa(n) > \frac{n}{\log n}$.
This continues to hold at least up to $n = 2\cdot 10^6$.

\begin{table}
\begin{center}
\begin{tabular}{|r|rrr|}
\hline
$n$   &$\pi(n)$   &$\kappa(n)$    &$\lfloor{\frac{n}{\log n}}\rceil$\\
\hline
100   
&25    
& 30   
&   22
\\
200   
&46    
& 45   
&   38
\\
300   
&62    
& 60   
&   53
\\
400   
&78    
& 75   
&   67\\
500   
&95    
&92    
&   80\\
600   
&109    
&106   
&  94 
\\
700   
& 125    
&119  
&   107\\
800   
&139    
&133   
&   120\\
900   
&154    
&141   
&   132\\
1000  
&168    
&157   
&   145\\
\hline
\end{tabular}
\end{center}
\label{table_one}
\caption{Gauss Circle Prime function $\kappa(n)$
compared with 
the prime number function $\pi(n)$
and the Prime Number Theorem asymptotic
$\frac{n}{\log n}$.}
\end{table}

\begin{table}
\begin{center}
\begin{tabular}{|r|rrr|}
\hline
$n$   &$\pi(n)$   &$\kappa(n)$    &$\lfloor{\frac{n}{\log n}}\rceil$\\
\hline
10000 
&1229    
&1188  
&   1086\\
20000 
&2262    
&2167  
&   2019\\
30000 
&3245    
&3138  
&   2910\\
40000 
&4203    
&4055  
&   3775\\
50000 
&5133    
&4922  
&   4621\\
100000  
&9592    
&9126   
&   8686\\
\hline
200000  
&17984    
&16963   
&   16385\\
300000  
&25997    
&24590   
&   23788\\
400000  
&33860    
&31988   
&   31010\\
500000  
&41538    
&39311  
&   38103\\
600000  
&49098    
&46687  
&   45097\\
700000  
&56543    
&53730  
& 52010\\
800000  
&63951    
&60772  
& 58857\\
900000  
&71274    
&67853  
&65645\\
$1 \times 10^6$   
& 78498  
&74854   
& 72382\\ \hline
$2 \times 10^6$   
&148933        
&143082   
&137849 \\ \hline
\end{tabular}
\label{table_two}
\caption{Values of $\pi(n)$,
$\kappa(n)$ and $\lfloor n/\log n \rceil$ for
$n  \leq 2 \times 10^6$.}
\end{center}
\end{table}

We next consider ratios of these values to further study their
behavior.  See Table 4.
These ratios show
$\kappa(n)$ is a better approximation for $\pi(n)$ than the
Prime Number Theorem for $n \leq 2 \times 10^6$,
and that all three of the functions are close together.

Observe some of the ratios fluctuate.  This is expected since the
prime numbers are not evenly distributed.

\begin{table}
\begin{center}
\begin{tabular}{|r|rrr|}
\hline
$n$   & $\frac{\pi(n)}{\kappa(n)}$    & $\frac{\pi(n)}{n/\log n}$ & $\frac{\kappa(n)}{n/\log n}$\\
\hline
100
&0.83333
&1.15129
&1.38155
\\
1000
&1.07006
&1.16050
&1.08451
\\
10000
&1.03451
&1.13195
&1.09418
\\
20000
&1.04383
&1.12008
&1.07304
\\
30000
&1.03409
&1.11508
&1.07831
\\
40000
&1.03649
&1.11344
&1.07423
\\
50000
&1.04286
&1.11075
&1.06509
\\
100000
&1.05106
&1.10431
&1.05066
\\
200000
&1.06018
&1.09757
&1.03525
\\
300000
&1.05721
&1.09287
&1.03372
\\
400000
&1.05852
&1.09191
&1.03155
\\
500000
&1.05665
&1.09015
&1.03170
\\
$1 \times 10^6$
&1.04868
&1.08448
&1.03414
\\
$2 \times 10^6$
&1.04089
&1.08040
&1.03796\\
\hline
\end{tabular}
\caption{Ratios between
$\pi(n)$, $\kappa(n)$ and
$n/\log n$.}

\end{center}
\end{table}

\newpage

\section{Heuristic argument for behavior of $\pi(n)$ and $\kappa(n)$}

In this section, we provide a heuristic argument to
show that
$\kappa(n)$ and 
$\pi(n)$ have similar behavior.

\begin{heuristic}
$\kappa(n) \approx \pi(n)$
for all $n$.
\end{heuristic}
\begin{proof}
By the Prime Number Theorem, 
the probability that an integer $a$ is a prime is
$$
     P(\mbox{$a$ is a prime}) \approx \frac{1}{\log a}.
$$
Since half of the integers
are odd,
$$
     P(\mbox{$a$ is a prime given $a$ is odd}) \approx \frac{2}{\log a}.
$$

For a sequence of odd integers $a_1, a_2, \ldots, a_n$, 
the expected number of primes in
       this sequence is about
$$ 
   \mbox{\# primes in \{$a_1, a_2, \ldots, a_n$\}} 
   \approx \sum_{k=1}^n \frac{2}{\log a_k}.
$$
By Gauss, we know the number $C(n)$ is about $\pi n^2$,
that is,
$\log(C(k)) \approx \log \pi + 2 \log k$.
By Theorem~\ref{theorem_number_lattice_points}, each
$C(k)$ is odd.  Thus
\begin{eqnarray*}
   \kappa(n) &=& \mbox{the number of primes }\\
             & & \mbox{in the sequence $\{C(1), C(2), \ldots, C(n)\}$}\\
             &\approx &  \sum_{k=1}^n \frac{2}{\log (C(k))}\\ 
& \approx &
\sum_{k=1}^n \frac{2}{\log \pi + 2 \log k}\\
& = & 
\frac{2}{\log \pi} + \sum_{k=2}^n \frac{2}{\log \pi + 2 \log k}\\
& \approx&
\frac{1}{2}\sum_{k=2}^n \frac{2}{\log k}\\
& \approx &
\pi(n).
\end{eqnarray*}
\end{proof}

\newpage

\section{Concluding remarks}

There are some natural research directions to take
that involve other
behavior of prime numbers and the Gauss Circle Primes.

\begin{question}(Twin Gauss Circle Primes)
In 1844 Polignac
conjectured that there are infinitely-many
twin primes, that is, primes $p$ and $p+2$ that are
distance $2$ apart~\cite[Page 400]{Polignac}.
We conjecture the Gauss Circle Prime analogue:
there are infinitely many
cases when
$C(r)$ and $C(r+1)$
are both prime.
\end{question}

\begin{question}
(Infinitude of Gauss Circle Primes)
It is a classical result
of Euclid that 
there are infinitely-many primes~\cite[Book IX, Proposition 20]{Euclid}.
We conjecture there are
infinitely-many Gauss Circle primes.
\end{question}

\begin{question}
(Skewes constant and Gauss Circle Primes)
Recall the logarithmic  integral function
$$
   \li(x) = \lim_{\delta \rightarrow 0}
            \left\{ \int_0^{1 - \delta} \frac{du}{\log u}
              + 
               \int_{1+\delta}^x \frac{du}{\log u}
          \right\}.
$$
In 1914 Littlewood proved
that $\pi(n)$ and 
the logarithmic integral
$\li(x)$ cross infinitely-many
times~\cite{Littlewood}.
In 1933 Skewes~\cite{Skewes} showed that there is such a crossing for
some $x$ with $x <  10^{10^{10^{34}}}$,
assuming the Reimann hypothesis holds.   
It would be interesting to study the behavior of the functions
$\pi(n)$, $\kappa(n)$ and $\lfloor n/\log n \rceil$
near Skewe's constant.
\end{question}

\begin{question}
(Gauss $n$-Sphere Primes)
A natural extension of this research is to explore lattice points occurring
in an $n$-dimensional sphere
for $n \geq 3$.
\end{question}


\newcommand{\journal}[6]{{\sc #1,} #2, {\it #3} {\bf #4} (#5), #6.}
\newcommand{\preprint}[3]{{\sc #1,} #2, preprint #3.}
\newcommand{\book}[4]{{\sc #1,} #2, #3, #4.}
\newcommand{\collection}[6]{{\sc #1,}  #2, #3, in {\it #4}, #5, #6.}
\newcommand{\JCTA}{J.\ Combin.\ Theory Ser.\ A}
\newcommand{\arxiv}[3]{{\sc #1,} #2, {\tt #3}.}
\newcommand{\article}[3]{{\sc #1,} #2, {\tt #3}.}
\newcommand{\journalfive}[5]{{\sc #1,} #2, {\it #3}  (#4), #5.}

\vspace{1in}
\noindent
Thomas Ehrenborg, 
Department of Mathematics,
Cornell University,
Ithaca NY 14853, USA
{\tt tre26@cornell.edu}

\end{document}